\documentclass[12pt]{amsart}
\usepackage{epsfig,comment}
\usepackage{mathtools}
\usepackage[subnum]{cases}
\usepackage{enumerate}
\usepackage{bm}
\usepackage{amsmath,amssymb,mathrsfs,amsthm}
\usepackage{color}
\usepackage[bookmarks=true,
bookmarksnumbered=true, breaklinks=true,
pdfstartview=FitH, hyperfigures=false, hyperfootnotes=false,
plainpages=false, naturalnames=true,
colorlinks=true,pagebackref=true,
pdfpagelabels,colorlinks=true,
linkcolor=black,
citecolor=black,
filecolor=black,
urlcolor=black]{hyperref}
\usepackage{wasysym}
\usepackage[hang]{footmisc}
\usepackage{verbatim}

\newcommand{\E}{\mathbb{E}}

\newcommand{\R}{\mathbb{R}}
\newcommand{\Sp}{\mathbb{S}^{d-1}}

\newcommand{\NN}{N^{-\frac{2}{d-1}}}
\newcommand{\NNN}{N^{\frac{2}{d-1}}}

\DeclareMathOperator{\vol}{vol}

\DeclareMathOperator{\conv}{conv}

\renewcommand{\phi}{\varphi}

\newtheorem {theorem}{Theorem}[section]
\newtheorem {proposition}[theorem]{Proposition}
\newtheorem {lemma}[theorem]{Lemma}
\newtheorem {corollary}[theorem]{Corollary}

\theoremstyle{definition}

\newtheorem {remark}[theorem]{Remark}

\title[Simultaneous random polytope approximation]{Simultaneous random polytope approximation\\ characterizes the Euclidean ball}
\author{Steven Hoehner}
\date{\today}

\begin{document}
\setcounter{footnote}{0}
\begin{abstract}\noindent
In this work, we present a new geometric characterization of the Euclidean ball based on approximation by random polytopes. We investigate when a single boundary sampling density is asymptotically optimal for two random polytope approximation problems, and show that for a  convex body in $\R^d$ with $C^2$ boundary and everywhere positive principal curvatures, one fixed density minimizes the leading expected errors for any two distinct intrinsic volume differences of inscribed random polytopes if and only if the body is a Euclidean ball. We show that an analogous characterization holds for the volume and mean width differences of circumscribed random polytopes, and that uniform sampling gives the converse statement in both models.

We also determine the unique density minimizing the average of two normalized leading errors, and obtain an explicit positive gap for every nonball. The proofs combine classical random polytope asymptotics, H\"older's inequality, and a curvature rigidity argument based on the Hsiung--Minkowski identities and Newton--Maclaurin inequalities.
\end{abstract}
\maketitle
\renewcommand{\thefootnote}{}
\renewcommand{\thefootnote}{\arabic{footnote}}
\setcounter{footnote}{0}

\section{Introduction}

The approximation of convex bodies by random polytopes is a central topic in stochastic convex geometry. In the inscribed model, points are sampled independently according to a common density from the boundary of a convex body according to a prescribed density function, and their convex hull approximates the convex body. Reitzner \cite{Reitzner2002} established an asymptotic formula for the expected intrinsic volume difference of a smooth convex body with positive Gauss curvature and an inscribed polytope (see Theorem~\ref{thm:Reitzner} below). In the circumscribed model, where boundary points are sampled independently and the supporting halfspaces at those points are intersected, B\"or\"oczky and Reitzner \cite{BoroczkyReitzner04} proved corresponding volume and mean width formulas for random circumscribed polytopes (see Theorem~\ref{thm:BR} below).

In each model, the expected error is governed by a curvature integral weighted by a negative power of the sampling density. In general, different intrinsic volumes have different minimizing densities, which are explicitly known. The purpose of this paper is to determine when a single density can achieve two individual minima \emph{simultaneously}, and to provide an explicit positive gap which shows this phenomenon is impossible for bodies which are not Euclidean balls.

The main idea is as follows. Equality in H\"older's inequality identifies the unique individual minimizing densities for each intrinsic volume difference. If two of them coincide, then their corresponding normalized elementary symmetric curvature functions must have a constant ratio. The Hsiung--Minkowski identities and the Newton--Maclaurin inequalities then force the curvature to be a constant and thus the body's boundary must be a sphere. Conversely,  every curvature weight is constant on a sphere, so the uniform sampling density attains all individual minima simultaneously. 


\subsection{Background and notation}

Throughout the paper, we let $K\subset\R^d$, $d\geq 2$, denote a convex body of class $C_+^2$. This means that $K$ is a convex, compact set with nonempty interior, its boundary $\partial K$ is $C^2$, and its principal curvatures $k_1,\ldots,k_{d-1}$ are everywhere positive. Let $S_K$ denote the restriction of $(d-1)$-dimensional
Hausdorff measure to $\partial K$, and let $|\partial K|:=S_K(\partial K)$ denote the surface area of $K$. For $m\in\{1,\ldots,d-1\}$, we also let $H_m$ be the normalized $m$th elementary symmetric function of the principal curvatures of $K$, given by
\[
 H_0:=1,\qquad
 H_m:=\binom{d-1}{m}^{-1}
       \sum_{1\leq i_1<\cdots<i_m\leq d-1}k_{i_1}\cdots k_{i_m}.
\]
In particular, $H_{d-1}=\kappa$ is the Gauss curvature of $K$.

Let $B_d$ denote the $d$-dimensional Euclidean unit ball. The intrinsic volumes of $K$ are defined by Steiner's formula
\[
 \vol_d(K+tB_d)=\sum_{j=0}^d\vol_{d-j}(B_{d-j})V_j(K)t^{d-j},
 \quad t\geq 0,
\]
where $\vol_0(B_0):=1$. In particular, $V_d(K)=\vol_d(K)$ is volume of $K$, $V_{d-1}$ is half of surface area of $K$, and $V_1$ is a positive dimensional constant times mean width of $K$. Recall that the mean width is defined by
\[
 w(K):=\frac{2}{d\vol_d(B_d)}\int_{\Sp}h_K(u)\,d\sigma(u),
\]
where $\sigma$ is the unnormalized spherical surface measure and $h_K(u):=\max_{x\in K}\langle x,u\rangle$ is the support function of $K$ in the direction $u\in\Sp$. For more background on convex bodies, we refer the reader to, e.g., the book \cite{SchneiderBook} of Schneider.

For an inscribed polytope $P\subset K$, the \emph{$j$th intrinsic volume difference} is defined by $\Delta_j(K,P):=V_j(K)-V_j(P)$. For a circumscribed polytope $Q\supset K$, $\delta_V(K,Q):=\vol_d(Q)-\vol_d(K)$ and $\delta_W(K,Q):=w(Q)-w(K)$ denote the \emph{symmetric volume difference} and \emph{mean width difference}, respectively. For more background on best and random approximations of convex bodies by inscribed and circumscribed polytopes, we refer the reader to, e.g., \cite{BFV2010,BH-2024,BHK,BoroczkyCsikos,BFH2013,BoroczkyReitzner04,glasgrub,Gruber93,Reitzner2002,SW2003} and the references therein.

\section{Main results}

\subsection{Simultaneous inscribed random approximation}

Let $C(\partial K)$ denote the space of continuous real-valued functions on $\partial K$. We shall work with the class of positive densities given by
\begin{equation}\label{eq:density-class}
 \mathscr{D}(\partial K):=\left\{\rho\in C(\partial K):\rho>0,
\, \int_{\partial K}\rho\,dS_K=1\right\}.
\end{equation}
For $\rho\in\mathscr D(\partial K)$, choose independent points $X_1,\ldots,X_N$ with common distribution $\rho\,dS_K$, and set
\[
 P_N(\rho):=\conv\{X_1,\ldots,X_N\}.
\]
Since $P_N(\rho)\subset K$ almost surely, the expected $j$th intrinsic volume difference is
\[
 \E[\Delta_j(K,P_N(\rho))]=V_j(K)-\E[V_j(P_N(\rho))].
\]
One of the main ingredients we will need is the following asymptotic formula of Reitzner \cite[Theorem~1]{Reitzner2002} that was mentioned in the introduction (see also \cite[Theorem~1.1]{BFH2013}). 

\begin{theorem}[Reitzner]\label{thm:Reitzner}
Let $j\in\{1,\ldots,d\}$ and let $\rho\in\mathscr{D}(\partial K)$ be fixed. Then
\begin{equation}\label{eq:Reitzner-asymptotic}
 \E[\Delta_j(K,P_N(\rho))]
 =c_{d,j}N^{-\frac{2}{d-1}}\int_{\partial K}\kappa(x)^{\frac{1}{d-1}}H_{d-j}(x)\rho(x)^{-\frac{2}{d-1}}\,dS_K(x) +o_\rho(N^{-\frac{2}{d-1}}),
\end{equation}
where $c_{d,j}>0$ is a constant depending only on $d$ and $j$.
\end{theorem}
\noindent Here the notation $o_\rho$ means that the density is fixed. 

\vspace{2mm}

For $1\leq j\leq d$, we denote the curvature functionals arising in Theorem~\ref{thm:Reitzner} and their corresponding minima by
\begin{equation}\label{eq:Ij}
 \mathcal{I}_j(K,\rho):=\int_{\partial K}\kappa(x)^{\frac{1}{d-1}}H_{d-j}(x)\rho(x)^{-\frac{2}{d-1}}\,dS_K(x)
 \qquad\text{and}\qquad 
 M_j(K):=\inf_{\eta\in\mathscr{D}(\partial K)}\mathcal{I}_j(K,\eta),
\end{equation}
respectively. Then Theorem~\ref{thm:Reitzner} states that $\E\Delta_j(K,P_N(\rho))\sim c_{d,j}\mathcal I_j(K,\rho)N^{-\frac{2}{d-1}}$ for each fixed density $\rho\in\mathscr{D}(\partial K)$. We say that a density $\rho\in\mathscr{D}(\partial K)$ is  \emph{asymptotically optimal} for this expected error if $\mathcal I_j(K,\rho)=M_j(K)$. The individual minima will be shown below to be attained uniquely. 

In what follows, the dimension $d$, the body $K$, and the sampling density $\rho$ are fixed as $N\to\infty$. For our purposes, optimality refers to minimization of the
leading expected error among densities in $\mathscr{D}(\partial K)$. Note that in this paper, we do not optimize finite-$N$ expectations over densities that may depend on $N$.

\vspace{2mm}

Our first main result is a characterization result for the Euclidean ball via simultaneous intrinsic volume approximation by inscribed polytopes.

\begin{theorem}\label{thm:main-random-inscribed}
Fix any two distinct indices $1\leq j_1<j_2\leq d$. The following are equivalent:
\begin{enumerate}
\item[(i)] $K$ is a Euclidean ball, up to translation.

\item[(ii)] There exists one density $\rho\in\mathscr{D}(\partial K)$ satisfying
$\mathcal{I}_{j_i}(K,\rho)=M_{j_i}(K)$ for $i\in\{1,2\}$.

\item[(iii)] There exists one fixed density $\rho\in\mathscr{D}(\partial K)$ such that
\[
 \lim_{N\to\infty}
 \frac{N^{\frac{2}{d-1}}\E[\Delta_{j_i}(K,P_N(\rho))]}{c_{d,j_i}M_{j_i}(K)}=1
 \quad \text{for } i\in\{1,2\}.
\]
\end{enumerate}
If these conditions hold, then the unique common density is $\rho=|\partial K|^{-1}\mathbf{1}_{\partial K}$, and it minimizes the leading expected error term for every $j\in\{1,\ldots,d\}$.
\end{theorem}

\subsection{Simultaneous circumscribed random approximation}

Our next main result concerns the circumscribed model of random approximation. To ensure that the random polyhedron is bounded, and is thus a polytope with finite expectation, we first fix a cube $Q$ containing $K$ in its interior. Let $\nu_K(x)$ denote the outer unit normal of $K$ at $x\in\partial K$, and let
\[
 H_K^+(x):=\{z\in\R^d:\langle z,\nu_K(x)\rangle
\leq\langle x,\nu_K(x)\rangle\}
\]
be the closed supporting  halfspace of $K$ at $x$ which contains $K$. 
For points $X_1,\ldots,X_N$ sampled independently from $\partial K$ according to the common distribution $\phi\,dS_K$, define the random polytope
\[
 P_N^Q(\phi):=Q\cap\bigcap_{i=1}^N H_K^+(X_i).
\]

The next result we will need is due to B\"or\"oczky and Reitzner \cite[Theorems~1 and~3]{BoroczkyReitzner04}.
\begin{theorem}[B\"or\"oczky--Reitzner]\label{thm:BR}
Let $\phi\in\mathscr D(\partial K)$ be fixed. Then there are constants $a_d,b_d>0$ depending only on $d$, such that
\begin{align}
 \E[\delta_V(K,P_N^Q(\phi))]&= a_d\NN\int_{\partial K}\kappa(x)^{\frac{1}{d-1}}\phi(x)^{-\frac{2}{d-1}}\,dS_K(x)+o_\phi(N^{-\frac{2}{d-1}}),
 \label{eq:BR-volume}\\
 \E[\delta_W(K,P_N^Q(\phi))]&=b_d\NN\int_{\partial K}\kappa(x)^{\frac{d}{d-1}}\phi(x)^{-\frac{2}{d-1}}\,dS_K(x)+o_\phi(N^{-\frac{2}{d-1}}).
 \label{eq:BR-width}
\end{align}
\end{theorem}
In particular, the leading coefficients do not depend on the fixed cube $Q$ containing $K$. Thus, for the circumscribed model, the relevant functionals are
\begin{equation}\label{eq:JV-JW}
\begin{aligned}
 \mathcal{J}_V(K,\phi)&:=\int_{\partial K}\kappa(x)^{\frac{1}{d-1}}\phi(x)^{-\frac{2}{d-1}}\,dS_K(x)\\
 \mathcal{J}_W(K,\phi)&:=\int_{\partial K}\kappa(x)^{\frac{d}{d-1}}\phi(x)^{-\frac{2}{d-1}}\,dS_K(x).
 \end{aligned}
\end{equation}
We also set $M_V^{\mathrm{out}}(K):=\inf_{\eta\in \mathscr{D}(\partial K)}\mathcal{J}_V(K,\eta)$, 
$M_W^{\mathrm{out}}(K):=\inf_{\eta\in \mathscr{D}(\partial K)}\mathcal{J}_W(K,\eta)$,  
\[
 A_V(\partial K):=\int_{\partial K}\kappa(x)^{\frac{1}{d+1}}\,dS_K(x),
 \quad\text{and}\quad 
 A_W(\partial K):=\int_{\partial K}\kappa(x)^{\frac{d}{d+1}}\,dS_K(x).
\]

Our next result concerns simultaneous random circumscribed approximation for the volume and mean width differences. 
\begin{theorem}\label{thm:main-random-circumscribed}
The following are equivalent:
\begin{enumerate}
\item[(i)] $K$ is a Euclidean ball, up to translation.
\item[(ii)] There exists one density $\phi\in\mathscr{D}(\partial K)$ minimizing both $\mathcal{J}_V(K,\cdot)$ and $\mathcal{J}_W(K,\cdot)$.
\item[(iii)] There exists one fixed density $\phi\in\mathscr{D}(\partial K)$ such that
\[
 \lim_{N\to\infty}\frac{N^{\frac{2}{d-1}}\E[\delta_V(K,P_N^Q(\phi))]}{a_d M_V^{\mathrm{out}}(K)}=1
 \quad\text{and}\quad
 \lim_{N\to\infty}\frac{N^{\frac{2}{d-1}}\E[\delta_W(K,P_N^Q(\phi))]}{b_d M_W^{\mathrm{out}}(K)}=1,
\]
where $a_d,b_d>0$ are the constants in Theorem~\ref{thm:BR}.
\end{enumerate}
If these conditions hold, then the unique common density is uniform. Moreover, the characterization is independent of the fixed cube $Q$ containing $K$ in its interior.
\end{theorem}

For every $C_+^2$ convex body which is not a Euclidean ball, at the end of the paper we will also show that there must be a positive relative gap, and determine the density minimizing the average of the two normalized leading errors; see Theorem~\ref{thm:random-compromise} below. 

\begin{remark}
    Besau, Hoehner, and Kur \cite{BHK} studied the asymptotic best and random approximation of the Euclidean ball with respect to inscribed and circumscribed polytopes. They showed that, remarkably, for each model there is a single (deterministic) polytope which is asymptotically best for all intrinsic volumes simultaneously, up to absolute constants whose ratio tends to one. They also asked if this property characterizes the Euclidean ball. Thus, the main results of this paper may be regarded as giving an affirmative answer to the probabilistic analogue of their question.
\end{remark}


\section{Proofs of the main results}
\label{sec:random}

The proofs of Theorems~\ref{thm:main-random-inscribed} and \ref{thm:main-random-circumscribed} will require some more background and lemmas.

\subsection{The individual optimal densities}

The first ingredient we need is the minimization argument for the curvature-weighted functionals that arise in Theorems~\ref{thm:Reitzner} and \ref{thm:BR} above. The type of  argument we will present is classical in the random polytopes literature; see, for example, \cite{Reitzner2002,SW2003}. We include the details for the reader's convenience. 

\begin{lemma}\label{lem:weighted-holder}
Let $w$ be a positive continuous function on $\partial K$, and define the $w$-weighted functionals
\[
 I_w(\rho):=\int_{\partial K}w(x)\rho(x)^{-\frac{2}{d-1}}\,dS_K(x)\quad\text{and} 
 \quad Z_w:=\int_{\partial K}w(x)^{\frac{d-1}{d+1}}\,dS_K(x).
\]
Then for every $\rho\in\mathscr{D}(\partial K)$,
\begin{equation}\label{eq:holder-minimum}
 I_w(\rho)\geq Z_w^{\frac{d+1}{d-1}}
\end{equation}
with equality if and only if
\begin{equation}\label{eq:holder-minimizer}
 \rho=\frac{w^{\frac{d-1}{d+1}}}{Z_w}\mathbf{1}_{\partial K}.
\end{equation}
In particular, the minimum is attained by a unique density in $\mathscr{D}(\partial K)$.
\end{lemma}

\begin{proof}
 Set $\alpha:=2/(d-1)$ and $p:=1+\alpha=(d+1)/(d-1)$. Note that since $p=1+\alpha$, we have $w^{1/p}=\rho^{\alpha/p}(w\rho^{-\alpha})^{1/p}$. Applying H\"older's inequality with the conjugate exponents $p/\alpha$ and $p$, we get
\[
 Z_w\leq\left(\int_{\partial K}\rho(x)\,dS_K(x)\right)^{\alpha/p}
 \left(\int_{\partial K}w(x)\rho(x)^{-\alpha}\,dS_K(x)\right)^{1/p}
 =I_w(\rho)^{1/p}.
\]
This proves \eqref{eq:holder-minimum}. Equality holds if and only if $\rho$ and $w\rho^{-\alpha}$ are proportional almost everywhere, i.e., when $\rho^p$ is proportional to $w$. Normalizing yields the proportionality factor and hence  \eqref{eq:holder-minimizer}. Conversely, substituting that density, we obtain $I_w=Z_w^p$. It is positive and continuous, so it is a member of  $\mathscr{D}(\partial K)$. Since densities in this class are continuous, almost everywhere equality implies everywhere equality, proving the uniqueness.
\end{proof}

For the inscribed model, define
\begin{equation}\label{eq:Aj-random}
 \mathcal{A}_j(K):=\int_{\partial K}
 \kappa(x)^{\frac{1}{d+1}}H_{d-j}(x)^{\frac{d-1}{d+1}}\,dS_K(x).
\end{equation}
Applying Lemma~\ref{lem:weighted-holder} with the curvature-weight function $w_j=\kappa^{1/(d-1)}H_{d-j}$, we derive that the minimizing density is given by
\begin{equation}\label{eq:rho-j-star}
 \rho_j^{\min}(x):=\frac{\kappa(x)^{\frac{1}{d+1}}H_{d-j}(x)^{\frac{d-1}{d+1}}}{\mathcal{A}_j(K)}\mathbf{1}_{\partial K}(x)
\end{equation}
and $M_j(K)=\mathcal{A}_j(K)^{\frac{d+1}{d-1}}$. Since all principal curvatures are positive and continuous functions on $\partial K$, these densities are also positive and continuous.

For the circumscribed model, applying Lemma~\ref{lem:weighted-holder} again yields the optimal densities
\begin{align}
 \varphi_V^{\min}(x)&=\frac{\kappa(x)^{1/(d+1)}}{A_V(\partial K)}\mathbf{1}_{\partial K}(x),&
 M_V^{\mathrm{out}}(K)&=A_V(\partial K)^{\frac{d+1}{d-1}};
 \label{eq:phiV}\\
 \varphi_W^{\min}(x)&=\frac{\kappa(x)^{d/(d+1)}}{A_W(\partial K)}\mathbf{1}_{\partial K}(x),&
 M_W^{\mathrm{out}}(K)&=A_W(\partial K)^{\frac{d+1}{d-1}}.
 \label{eq:phiW}
\end{align}

\subsection{A curvature rigidity lemma}

To prove the probabilistic rigidity theorems, we will need one final ingredient. The smooth version of the following curvature ratio characterization
is due to Koh \cite{Koh2000}. We include a proof under the present $C_+^2$ assumption of this paper, using the Hsiung--Minkowski identities and the Newton--Maclaurin inequalities.

\begin{lemma}
\label{lem:curvature-quotient}
Let $\Sigma\subset\R^{n+1}$ be a compact, connected, embedded $C_+^2$ strictly convex hypersurface. Let $H_m$, $m\in\{0,1,\ldots,n\}$, be the normalized elementary symmetric functions of its principal curvatures, with $H_0\equiv 1$. Suppose that for some $0\leq k<\ell\leq n$ and some constant $c>0$, we have
\begin{equation}\label{eq:Hk-cell}
 H_k=cH_\ell
 \qquad\text{on }\Sigma.
\end{equation}
Then $\Sigma$ is a round sphere.
\end{lemma}

To prove the lemma, we will use the Hsiung--Minkowski identities and the equality cases in the Newton--Maclaurin inequalities. We include the details for the reader's convenience.

\begin{proof}
If $n=1$, then $k=0$ and $\ell=1$. In this case, the relation \eqref{eq:Hk-cell} says that the curvature of the closed strictly convex plane curve $\Sigma$ is constant. The tangent angle therefore changes at constant speed with respect to arclength, and integration of the Frenet equations shows that $\Sigma$ is a circle. Thus, we may assume that $n\geq 2$.

Translate the origin into the interior of the convex body enclosed by $\Sigma$. Let $X:\Sigma\to\R^{n+1}$ be the position vector, let $\nu$ be the outer unit normal, and set $h:=\langle X,\nu\rangle$. 
Note that $h>0$ on $\Sigma$. For $j\in\{1,\ldots,n\}$, the Hsiung--Minkowski identity states that (see \cite{Hsiung})
\begin{equation}\label{eq:Hsiung-Minkowski}
 \int_\Sigma H_{j-1}\,dS
 =\int_\Sigma hH_j\,dS.
\end{equation}
For smooth, strictly convex hypersurfaces this is a classical result. In the present $C^2$ setting, choose smooth support functions $h_m$ converging to the support function of $\Sigma$ in $C^2(\mathbb S^n)$ and such that $\nabla_{\mathbb S^n}^2h_m+h_mI$ is positive definite. The curvature-radius matrices and their determinants converge uniformly and remain uniformly positive definite. Expressing both integrals in Gauss map coordinates therefore shows that \eqref{eq:Hsiung-Minkowski} passes to the limit. 

The Newton--Maclaurin inequalities for positive numbers (see \cite[Chapter~II]{Hardy-Littlewood-Polya}) imply
\begin{equation}\label{eq:newton-quotients}
 \frac{H_1}{H_0}\geq\frac{H_2}{H_1}\geq\cdots\geq\frac{H_n}{H_{n-1}}
\end{equation}
pointwise on $\Sigma$. Equality in any one of these Newton inequalities at a point with positive principal curvatures holds if and only if all principal curvatures are equal at that point.

Assume first that $k\geq 1$. Applying \eqref{eq:Hsiung-Minkowski} with $j=k$ and $j=\ell$, and then using \eqref{eq:Hk-cell}, we get
\begin{align}
 \int_\Sigma H_{k-1}\,dS
 &=\int_\Sigma hH_k\,dS
 =c\int_\Sigma hH_\ell\,dS
 =c\int_\Sigma H_{\ell-1}\,dS.
 \label{eq:integral-H-equality}
\end{align}
On the other hand, by \eqref{eq:newton-quotients},
\[
 \frac{H_k}{H_{k-1}}
 \geq
 \frac{H_\ell}{H_{\ell-1}}.
\]
Using $H_k=cH_\ell$ and the positivity of $H_\ell$, we obtain the pointwise inequality
\begin{equation}\label{eq:pointwise-H-equality}
 H_{k-1}\leq cH_{\ell-1}.
\end{equation}
The integral of the nonnegative continuous function $cH_{\ell-1}-H_{k-1}$ is zero by \eqref{eq:integral-H-equality}. Therefore
\begin{equation}\label{eq:pointwise-equality-all}
 H_{k-1}=cH_{\ell-1}
 \qquad\text{at every point of }\Sigma.
\end{equation}
Combining \eqref{eq:Hk-cell} and \eqref{eq:pointwise-equality-all}, we get
\[
 \frac{H_k}{H_{k-1}}=\frac{H_\ell}{H_{\ell-1}}.
\]
By the monotone chain \eqref{eq:newton-quotients}, all the intervening quotients must be equal. Hence equality holds in a Newton inequality at every point, and all principal curvatures are equal at every point.

It remains to consider the case $k=0$. Since $H_0\equiv 1$, equation \eqref{eq:Hk-cell} says that
\begin{equation}\label{eq:Hell-constant}
 H_\ell=C:=c^{-1}.
\end{equation}

\noindent Suppose first that $\ell=1$. Then $H_1\equiv C$. By \eqref{eq:newton-quotients}, we get $H_2\leq H_1^2=C^2$. 
Applying \eqref{eq:Hsiung-Minkowski} with $j=1$ and $j=2$, respectively, we obtain
\[
 \int_\Sigma H_0\,dS=C\int_\Sigma h\,dS
\]
and
\[
 C\int_\Sigma H_0\,dS
 =\int_\Sigma H_1\,dS
 =\int_\Sigma hH_2\,dS
 \leq C^2\int_\Sigma h\,dS
 =C\int_\Sigma H_0\,dS.
\]
Thus equality holds in the inequality above. Since $h>0$ and $C^2-H_2\geq0$ is continuous, it follows that $H_2=C^2=H_1^2$ at every point of $\Sigma$. Hence equality holds in the first Newton inequality, and all principal curvatures are equal at every point.

Now suppose that $\ell\geq2$. The Maclaurin inequalities give
\begin{equation}\label{eq:Maclaurin-two}
 H_1\geq C^{1/\ell}
 \quad\text{and}\quad
 H_{\ell-1}\geq C^{(\ell-1)/\ell}.
\end{equation}
Using \eqref{eq:Hsiung-Minkowski}, \eqref{eq:Hell-constant}, the positivity of $h$, and \eqref{eq:Maclaurin-two}, we thus obtain
\begin{align*}
 \int_\Sigma H_{\ell-1}\,dS
 &=\int_\Sigma hH_\ell\,dS
 =C\int_\Sigma h\,dS
 \leq C^{(\ell-1)/\ell}\int_\Sigma hH_1\,dS\\
 &=C^{(\ell-1)/\ell}\int_\Sigma H_0\,dS
 \leq\int_\Sigma H_{\ell-1}\,dS.
\end{align*}
Therefore, equality holds in every inequality. In particular,
\[
 \int_\Sigma h\left(C^{(\ell-1)/\ell}H_1-C\right)dS=0.
\]
The integrand is continuous and nonnegative by \eqref{eq:Maclaurin-two}, and $h>0$. Hence $H_1=C^{1/\ell}=H_\ell^{1/\ell}$ at every point of $\Sigma$. Equality in the Maclaurin inequality therefore implies that all principal curvatures are equal at every point.

We have shown in all cases that $\Sigma$ is totally umbilic. Let $\lambda(x)>0$ denote the common principal curvature at $x\in\Sigma$. Then $H_m(x)=\lambda(x)^m$ for every $m$, and \eqref{eq:Hk-cell} gives
$\lambda(x)^k=c\lambda(x)^\ell$. Since $\lambda(x)>0$, we thus obtain $\lambda(x)^{\ell-k}=c^{-1}$, 
so $\lambda$ is a positive constant on $\Sigma$. Using the convex-geometric sign convention for the shape operator, $d\nu=\lambda\,dX$. Hence $d(X-\lambda^{-1}\nu)=0$. Since $\Sigma$ is connected, there exists $a\in\R^{n+1}$ such that $X-\lambda^{-1}\nu=a$ on $\Sigma$. 
Therefore, $X=a+\lambda^{-1}\nu$ on $\Sigma$. Every unit vector occurs as an outer normal, since the enclosed convex body has a support point in every direction. Hence $\Sigma=a+\lambda^{-1}\mathbb S^n$, which is a round sphere.
\end{proof}


\begin{corollary}\label{lem:constant-curvature}
If $K$ has constant Gauss curvature, then it is a Euclidean ball, up to translation.
\end{corollary}

\begin{proof}
Apply Lemma~\ref{lem:curvature-quotient} to $\Sigma=\partial K$, with $n=d-1$, $k=0$ and $\ell=d-1$. The hypothesis holds because $H_0=1$ and $H_{d-1}=\kappa$ is a positive constant.
\end{proof}

\subsection{Proofs of the main theorems}

\begin{proof}[Proof of Theorem~\ref{thm:main-random-inscribed}]
By Theorem~\ref{thm:Reitzner}, for every fixed density $\rho\in\mathscr{D}(\partial K)$ we have 
\begin{equation}\label{eq:random-normalized-limit}
 \lim_{N\to\infty}\frac{N^{\frac{2}{d-1}}\E[\Delta_j(K,P_N(\rho))]}{c_{d,j}M_j(K)}
 =\frac{\mathcal I_j(K,\rho)}{M_j(K)}.
\end{equation}
Therefore, (ii) and (iii) are equivalent.

Suppose that (ii) holds. By the uniqueness in Lemma~\ref{lem:weighted-holder}, we have $\rho=\rho_{j_1}^{\min}=\rho_{j_2}^{\min}$. By \eqref{eq:rho-j-star}, canceling $\kappa^{1/(d+1)}$ and raising to the power $(d+1)/(d-1)$, we get
\[
 \frac{H_{d-j_1}}{\mathcal{A}_{j_1}(K)^{\frac{d+1}{d-1}}}
 =\frac{H_{d-j_2}}{\mathcal{A}_{j_2}(K)^{\frac{d+1}{d-1}}}.
\]
Hence $H_{d-j_2}=cH_{d-j_1}$ for a positive constant $c$. The indices also satisfy $0\leq d-j_2<d-j_1\leq d-1$. Lemma~\ref{lem:curvature-quotient} now implies that $\partial K$ is a round sphere, which proves (i).

Conversely, if $K$ is a translate of $rB_d$, then all of its principal curvatures equal $1/r$. Every weight $w_j$ is then a constant, so by \eqref{eq:rho-j-star} we have $\rho_j^{\min}=|\partial K|^{-1}\mathbf{1}_{\partial K}$ for every $j$. This uniform density provides one common sampling model for all intrinsic volumes and proves (ii). Moreover, the individual uniqueness of the minimizing densities proves uniqueness of the common density. This completes all implications.
\end{proof}

\begin{proof}[Proof of Theorem~\ref{thm:main-random-circumscribed}]
By Theorem~\ref{thm:BR}, the two limits in (iii) equal $\mathcal J_V(K,\phi)/M_V^{\mathrm{out}}(K)$ and $\mathcal J_W(K,\phi)/M_W^{\mathrm{out}}(K)$. Hence (ii) and (iii) are equivalent. If (ii) holds, then by uniqueness, we must have $\phi=\varphi_V^{\min}=\varphi_W^{\min}$. Comparing \eqref{eq:phiV} and \eqref{eq:phiW}, we deduce that $\kappa^{(d-1)/(d+1)}$ is constant. Thus, Corollary~\ref{lem:constant-curvature} proves (i). Conversely, on a ball both densities in \eqref{eq:phiV} and \eqref{eq:phiW} are uniform, which proves (ii) and the uniqueness assertion. The leading terms in Theorem~\ref{thm:BR} are independent of $Q$, so the characterization is independent of $Q$ as well.
\end{proof}

\begin{remark}\label{rem:fixed-benchmark}
The coefficient optimality condition can equivalently be written as an expectation ratio with the fixed optimal density:
\[
 \frac{\E[\Delta_j(K,P_N(\rho))]}{\E[\Delta_j(K,P_N(\rho_j^{\min}))]}\longrightarrow 1 \quad\text{as }N\to\infty.
\]
Indeed, the ratio tends to $\mathcal I_j(K,\rho)/M_j(K)$ by Theorem~\ref{thm:Reitzner}. The analogous outer ratios use $\varphi_V^{\min}$ and $\varphi_W^{\min}$ instead. Only two fixed densities are compared, so a uniform asymptotic estimate is not needed. On a ball,  the common uniform density equals every fixed minimizing density, and these ratios are identically one.
\end{remark}

\subsection{On the positive gap for nonballs}

We want to show that there is a positive gap for convex bodies which are not Euclidean balls. First, we determine the best common density for the average of two normalized coefficient errors. We will state the underlying variational result in a form which handles both random models at once.

\begin{proposition}\label{prop:compromise}
Let $f_1,f_2\in\mathscr{D}(\partial K)$, and define
\[
 R_i(K,\rho):=\int_{\partial K}f_i(x)^{\frac{d+1}{d-1}}\rho(x)^{-\frac{2}{d-1}}\,dS_K(x) \quad \text{for }i\in\{1,2\},
\]
and
\begin{equation}\label{eq:Z-compromise}
 Z:=\int_{\partial K}\left(\frac{f_1(x)^{\frac{d+1}{d-1}}+f_2(x)^{\frac{d+1}{d-1}}}{2}\right)^{\frac{d-1}{d+1}}\,dS_K(x).
\end{equation}
Then
\begin{equation}\label{eq:compromise-minimum}
 \min_{\rho\in\mathscr{D}(\partial K)}\frac{R_1(K,\rho)+R_2(K,\rho)}2=Z^{\frac{d+1}{d-1}},
\end{equation}
and its unique minimizing density is
\begin{equation}\label{eq:compromise-density}
 \rho_{\mathrm{avg}}(x):=\frac{\left((f_1(x)^{\frac{d+1}{d-1}}+f_2(x)^{\frac{d+1}{d-1}})/2\right)^{\frac{d-1}{d+1}}}{Z}\mathbf{1}_{\partial K}(x).
\end{equation}
Moreover, $Z\geq 1$, with equality if and only if $f_1=f_2$. In particular,
\begin{equation}\label{eq:random-gap-functional}
 \max\{R_1(K,\rho),R_2(K,\rho)\}\geq Z^{\frac{d+1}{d-1}}
 \qquad\text{for every }\rho\in\mathscr{D}(\partial K).
\end{equation}
\end{proposition}

\begin{proof}
The average functional here equals $I_w(\rho)$ with $w=(f_1^{\frac{d+1}{d-1}}+f_2^{\frac{d+1}{d-1}})/2$. This weight is positive and continuous, so Lemma~\ref{lem:weighted-holder} yields the exact minimum and the unique density.

To prove the equality assertion for $Z$, we will apply H\"older's inequality with the density $\rho_{\mathrm{avg}}$ separately to the two weights $f_i^{\frac{d+1}{d-1}}$. Since $\int f_i\,dS_K=1$, we have  $R_i(K,\rho_{\mathrm{avg}})\geq 1$ for $i\in\{1,2\}$. Their average is $Z^{\frac{d+1}{d-1}}$ by \eqref{eq:compromise-minimum}, so $Z\geq 1$. If $Z=1$, then both terms equal one. The equality case in Lemma~\ref{lem:weighted-holder} yields $\rho_{\mathrm{avg}}=f_1$ and $\rho_{\mathrm{avg}}=f_2$. Hence $f_1=f_2$. Conversely, if the two densities coincide, then the integrand in \eqref{eq:Z-compromise} is that common density and $Z=1$. Finally, the maximum is at least the average, which is at least its minimum $Z^{\frac{d+1}{d-1}}$. This proves \eqref{eq:random-gap-functional}.
\end{proof}

\begin{theorem}
\label{thm:random-compromise}
For the inscribed model and fixed distinct indices $j_1$ and $j_2$, take $f_i=\rho_{j_i}^{\min}$ in Proposition~\ref{prop:compromise}, and denote the corresponding $Z$-value by $Z_{K,j_1,j_2}^{\mathrm{in}}$. Then
\[
 R_i(K,\rho)=\frac{\mathcal I_{j_i}(K,\rho)}{M_{j_i}(K)}.
\]
The density \eqref{eq:compromise-density} uniquely minimizes the average of these two normalized leading errors. If $K$ is not a ball, then
\begin{equation}\label{eq:eta-random-in}
 \eta_{K,j_1,j_2}^{\mathrm{rand,in}}:=(Z_{K,j_1,j_2}^{\mathrm{in}})^{\frac{d+1}{d-1}}-1>0,
\end{equation}
and every fixed density $\rho\in\mathscr{D}(\partial K)$ satisfies
\begin{equation}\label{eq:random-gap-expectation-in}
 \lim_{N\to\infty}\max_{1\leq i\leq 2}
 \frac{N^{\frac{2}{d-1}}\E[\Delta_{j_i}(K,P_N(\rho))]}{c_{d,j_i}M_{j_i}(K)}
 \geq1+\eta_{K,j_1,j_2}^{\mathrm{rand,in}}.
\end{equation}

For the circumscribed model, take $f_1=\varphi_V^{\min}$ and $f_2=\varphi_W^{\min}$, and denote the corresponding $Z$-value by $Z_K^{\mathrm{out}}$. If $K$ is not a ball, then
\begin{equation}
    \eta_K^{\mathrm{rand,out}}:=(Z_K^{\mathrm{out}})^{\frac{d+1}{d-1}}-1>0,
\end{equation}   
and every fixed density $\varphi\in\mathscr{D}(\partial K)$ satisfies
\begin{equation}
    \lim_{N\to\infty}\max\left\{\frac{\NNN\E[\delta_V(K,P_N^Q(\varphi)]}{a_d M_V^{\mathrm{out}}(K)},\frac{\NNN\E[\delta_W(K,P_N^Q(\varphi)]}{b_d M_W^{\mathrm{out}}(K)}\right\}\geq 1+\eta_K^{\mathrm{rand,out}}
\end{equation}
\end{theorem}

\begin{proof}
Since $M_j(K)=\mathcal A_j(K)^{\frac{d+1}{d-1}}$, by formula \eqref{eq:rho-j-star} we have
\[
 (\rho_j^{\min})^{\frac{d+1}{d-1}}=\frac{\kappa^{1/(d-1)}H_{d-j}}{M_j(K)}\mathbf{1}_{\partial K}.
\]
This proves the identity for $R_i$. If $K$ is not a ball, then by Theorem~\ref{thm:main-random-inscribed} we have $f_1\neq f_2$, so Proposition~\ref{prop:compromise} implies $Z>1$. Each normalized expectation converges to $R_i(K,\rho)$ by \eqref{eq:random-normalized-limit}. The maximum of a fixed finite number of convergent sequences also converges to the maximum of their limits. Now apply \eqref{eq:random-gap-functional} to obtain \eqref{eq:random-gap-expectation-in}.

Similarly, for the circumscribed model, using \eqref{eq:phiV} and\eqref{eq:phiW} we obtain
\[
 (\varphi_V^{\min})^{\frac{d+1}{d-1}}=\frac{\kappa^{1/(d-1)}}{M_V^{\mathrm{out}}(K)}\mathbf{1}_{\partial K}\quad\text{and}\quad 
 (\varphi_W^{\min})^{\frac{d+1}{d-1}}=\frac{\kappa^{d/(d-1)}}{M_W^{\mathrm{out}}(K)}\mathbf{1}_{\partial K}.
\] 
By Theorem~\ref{thm:main-random-circumscribed}, we have distinct densities precisely for nonballs. The proof then follows from Proposition~\ref{prop:compromise} and Theorem~\ref{thm:BR} in the same way.
\end{proof}

\begin{remark}
The density $\rho_{\mathrm{avg}}$ minimizes the average of
the two normalized leading errors. We are not claiming here that it minimizes their maximum, nor that $Z^{(d+1)/(d-1)}$ is the exact minimum of that maximum. Nevertheless, the lower bound applies to every fixed density.
\end{remark}


\bibliographystyle{plain}
\bibliography{main}

@article{BFV2010,
author={B\'ar\'any, I. and Fodor, F. and V\'igh, V.},
title={Intrinsic volumes of inscribed random polytopes in smooth convex bodies},
journal={Adv. in Appl. Probab.},
volume={42},
year={2010}, 
pages={605--619}
}

@article{BFH2013,
author={B\"or\"oczky, K. J. and Fodor, F. and Hug, D.},
title={Intrinsic volumes of random polytopes with vertices on the boundary of a convex body},
journal={Trans. Amer. Math. Soc.},
volume={365},
year={2013}, 
pages={785--809}
}

@article{BH-2024,
author={Besau, F. and Hoehner, S.},
year={2024},
title={An intrinsic volume metric for the class of convex bodies in $\mathbb{R}^n$},
journal={Communications in Contemporary Mathematics},
pages={2350006},
volume={26},
number={3}
}

@article{BHK,
  author  = {Besau, F. and Hoehner, S. and Kur, G.},
  title   = {Intrinsic and dual volume deviations of convex bodies and polytopes},
  journal = {International Mathematics Research Notices. IMRN},
  year    = {2021},
  volume  = {2021},
  number  = {22},
  pages   = {17456--17513}
}

@article{BoroczkyCsikos,
author={B\"or\"oczky, K. J. and Csik\'os, B.},
journal={Abhandlungen aus dem {M}athematischen {S}eminar der {U}niversit\"at {H}amburg},
title={Approximation of smooth convex bodies by circumscribed polytopes with respect to the surface area},
year={2009},
volume={79},
pages={229--264}
}

@article{BoroczkyReitzner04,
  author  = {B{\"o}r{\"o}czky, K. J. and Reitzner, M.},
  title   = {Approximation of smooth convex bodies by random circumscribed polytopes},
  journal = {The Annals of Applied Probability},
  year    = {2004},
  volume  = {14},
  number  = {1},
  pages   = {239--273}
}

@article{Gruber93,
  author  = {Gruber, P. M.},
  title   = {Asymptotic estimates for best and stepwise approximation of convex bodies. {II}},
  journal = {Forum Mathematicum},
  year    = {1993},
  volume  = {5},
  number  = {6},
  pages   = {521--538}
}

@article{glasgrub,
  author  = {Glasauer, S. and Gruber, P. M.},
  title   = {Asymptotic estimates for best and stepwise approximation of convex bodies. {III}},
  journal = {Forum Mathematicum},
  year    = {1997},
  volume  = {9},
  number  = {4},
  pages   = {383--404}
}

@book{Hardy-Littlewood-Polya,
author={Hardy, G. H. and Littlewood, J. E. and  P\'olya, G.},
title={Inequalities}, 
publisher={Cambridge University Press},
address={Cambridge},
edition={2}, 
year={1952}
}

@article{Hsiung,
author={Hsiung, C. C.},
title={Some integral formulas for closed hypersurfaces}, 
journal={Mathematica Scandinavica},
volume={2},
year={1954}, 
pages={286--294}
}

@article{Koh2000,
author={Koh, S.-E.},
title={Sphere theorem by means of the ratio of mean curvature functions}, journal={Glasgow Mathematical Journal}, year={2000},
volume={42},
number={1},
pages={91--95} 
}

@article{Reitzner2002,
  author  = {Reitzner, M.},
  title   = {Random points on the boundary of smooth convex bodies},
  journal = {Transactions of the American Mathematical Society},
  year    = {2002},
  volume  = {354},
  number  = {6},
  pages   = {2243--2278}
}

@book{SchneiderBook,
  author    = {Schneider, R.},
  title     = {Convex Bodies: The Brunn--Minkowski Theory},
  edition   = {Second expanded},
  series    = {Encyclopedia of Mathematics and its Applications},
  volume    = {151},
  publisher = {Cambridge University Press},
  address   = {Cambridge},
  year      = {2014}
}

@incollection{SW2003,
author="Sch{\"u}tt, C.
and Werner, E.",
title={Polytopes with vertices chosen randomly from the boundary of a convex body},
series="Lecture Notes in Mathematics",
booktitle="Geometric Aspects of Functional Analysis: Israel Seminar 2001-2002",
year="2003",
publisher="Springer",
address="Berlin, Heidelberg",
volume="1807",
pages="241--422"
}

\vspace{3mm}

\noindent {\sc Department of Mathematics \& Computer Science, Longwood University,\\ 201 High St., Farmville, Virginia 23901}\\
{\it E-mail address:} {\tt hoehnersd@longwood.edu}

\end{document}